\documentclass[11pt,reqno]{amsart}
\usepackage{palatino}
\usepackage{amsfonts,amsmath}
\usepackage{graphicx,enumerate,soul}
\usepackage[hmarginratio=1:1,hscale=0.75]{geometry}
\usepackage[latin1]{inputenc}
\usepackage[all]{xy}
\usepackage{cite}

\usepackage{amsmath}
\usepackage{amssymb}
\usepackage{hyperref}
\usepackage{color}
\usepackage[T1]{fontenc}
\usepackage{lmodern}

\newtheorem{counter}{Counter}
\newtheorem{lem}[counter]{Lemma}
\newtheorem{defn}[counter]{Definition}
\newtheorem{thm}[counter]{Theorem}
\newtheorem{prop}[counter]{Proposition}

\newcommand{\R}{\mathbb{R}}

\renewcommand{\L}{\mathcal{L}}

\renewcommand{\lg}{\langle}
\newcommand{\rg}{\rangle}
\newcommand{\lra}{\longrightarrow}

 \newcommand{\sse}{\subseteq}

\newcommand{\pd}{\partial}
\newcommand{\rhu}{\rightharpoonup}

\newcommand{\fal}{\forall}
\newcommand{\8}{\infty}

\newcommand{\vph}{\varphi}
\newcommand{\vep}{\varepsilon} 

\newcommand{\dt}{\delta}

 \newcommand{\ggm}{\Gamma}

\DeclareMathOperator{\pero}{{\mathbb{T}_0}}

{   \end{list} }

\definecolor{mygreen}{rgb}{0.1,0.75,0.2}

\renewcommand{\d}{\,{\operatorname{d}}}

\newcommand{\I}{{\mathbb{T}}}
\newcommand{\tdV}{\widetilde{V}}

\begin{document}

\title[Maximal monotone operator in non-reflexive Banach space]{Maximal monotone operator theory and its applications to thin film equation in epitaxial growth on vicinal surface}
\author{Yuan Gao}
\address{Department of Mathematics
Hong Kong University of Science and Technology, Clear Water Bay, Kowloon, Hong Kong}
\email{gaoyuan12@fudan.edu.cn; maygao@ust.hk}
\author{Jian-Guo Liu}
\address{Department of Mathematics and Department of
  Physics\\Duke University,
  Durham NC 27708, USA}
\email{jliu@phy.duke.edu}
\author{Xin Yang Lu}
\address{Department of Mathematical Sciences
Lakehead University, Thunder Bay, ON, P7B 5E1, Canada}
\email{xlu8@lakekeadu.ca; xinyang.lu@mcgill.ca}

\author{Xiangsheng Xu}
\address{Department of Mathematics and Statistics\\Mississippi State University,
Mississippi State, MS 39762, USA}
\email{xxu@math.msstate.edu}


\date{\today}

\keywords{Fourth-order degenerate parabolic equation, non-reflexive Banach space, Radon measure, global strong solution, hidden singularity, uniqueness of sub-differential }
\begin{abstract}
  In this work we consider
  \begin{equation}\label{abs}
w_t=[(w_{hh}+c_0)^{-3}]_{hh},\qquad w(0)=w^0,
\end{equation}
which is derived from a thin film equation for epitaxial growth on vicinal surface. We formulate the problem as the gradient flow of a suitably-defined convex functional in a non-reflexive space. Then by restricting it to a Hilbert space and proving the uniqueness of its sub-differential, we can apply the classical maximal monotone operator theory. The mathematical
difficulty is due to the fact that $w_{hh}$ can appear as a positive Radon measure.
  We prove the existence of a global
  strong solution with hidden singularity. In particular, \eqref{abs} holds almost everywhere when $w_{hh}$ is replaced by its absolutely continuous part.
\end{abstract}

\maketitle

 \section{Introduction}

 \subsection{Background and motivation.}
Below the roughening
transition temperature, the crystal surface is not smooth and forms steps, terraces and adatoms on the substrate, which form solid films. Adatoms
detach from steps, diffuse on the terraces until they meet other
steps and reattach again, which lead to a step flow on the crystal
surface. The evolution of individual steps is described mathematically
by the Burton-Cabrera-Frank (BCF) type discrete models \cite{BCF}. Although discrete models do have the advantage of reflecting physical principle directly,
when we study the evolution of crystal growth from macroscopic view, continuum approximation for the discrete models involves fewer variables than discrete models and can briefly show the evolution of step flow. Many interesting continuum models can be found in the literature on surface morphological evolution; see \cite{Zang1990,Tang1997,Yip2001,Xiang2002,Xiang2004,Shenoy2002,Margetis2011,Leoni2014,our1} for one dimensional models and \cite{Margetis2006,Xiang2009} for two dimensional models.
\textsc{Kohn} clarified the evolution of surface height from the thermodynamic viewpoint in
{the} book \cite{Kohnbook}. He considered the classical surface energy, which dates back to the pioneering work of Mullins \cite{1957} and Najafabadi, Srolovitz \cite{SSR1}, {   given by}
\begin{equation}
  F(h):=\int_{\Omega} {  (}\beta_1|\nabla h|+\beta_3|\nabla h|^3 {  )}\d x,
\end{equation}
where $\Omega$ is the ``step locations area'' we {are} concerned {with}.
Then, by conservation of mass, we have the equation for surface height $h$
\begin{equation}\label{heighteq}
\begin{aligned}
  h_t&=\nabla \Big(M(\nabla h) \nabla \frac{\delta F}{\delta h}\Big)\\
  &=-\nabla\bigg( M(\nabla h) \nabla\Big( \nabla\cdot \Big( \beta_1 \frac{\nabla h}{|\nabla h|}+\beta_3|\nabla h|\nabla h\Big)\Big)\bigg),
  \end{aligned}
\end{equation}
where $M(\nabla h)$ is a suitable ``mobility'' {term}
 depending on the dominating process of surface motion. Often two limit cases are considered. For diffusion-limited (DL) case, the dominated dynamics is diffusion across the terraces,
 {we have} $M(\nabla h)=1$; while for attachment-detachment-limited (ADL) case, the dominating
processes are the attachment and detachment of atoms at steps edges, {and} $M(\nabla h)=\frac{1}{|\nabla h|}$.
In the DL regime,   \cite{giga} obtained a fully understanding of the evolution and proved
the finite-time flattening. However, in the ADL regime, due to the difficulty brought by mobility
{   term} $M(\nabla h)=\frac{1}{|\nabla h|}$, the dynamics of the solution to surface height equation \eqref{heighteq}, with either $\beta_1=0$ or $\beta_1\neq 0$, is still an open question {  (see for instance \cite{Kohnbook})}.

Although a general surface may have peaks and valleys, the analysis of step motion on the level
of continuous PDE is complicated and we focus on a simpler situation first:
a monotone one-dimensional step train. In this simpler case, $\beta_1=0$, and by taking $\beta_3=\frac{1}{2}$, \eqref{heighteq} becomes
\begin{equation}
  h_t=-\Big[\frac{1}{h_x}\big(3h_xh_{xx}\big)_x\Big]_x.
\end{equation}
\textsc{Ozdemir, Zangwill} \cite{Zang1990} and \textsc{Al Hajj Shehadeh, Kohn and Weare} \cite{She2011} realized that using the step slope as a new variable is a convenient way to study the continuum PDE model, i.e.,
\begin{equation}\label{ueq}
  u_t=-u^2(u^3)_{hhhh}, \quad u(0)=u_0,
\end{equation}
where $u$, considered as a $[0,1)$-periodic function of the step height $h$, is the step
  slope of the surface.
  \cite{our1} provided a method to rigorously obtain the convergence rate of discrete model to its corresponding continuum limit.

Two questions then arise. One is how to formulate a proper solution to \eqref{ueq} and prove the well-posedness of its solution. The other one is the positivity of the solution. More explicitly, we want to know whether the sign of the solution $u$ to \eqref{ueq} is persistent.  Our goal in this work is to validate the continuum slope PDE \eqref{ueq} by answering the above two questions. The equation \eqref{ueq} is a degenerate equation and we cannot prevent $u$
from touching zero, where singularity arise. We observe that we are able to rewrite \eqref{ueq} as an abstract evolution equation with maximal monotone operator using $\frac{1}{u}$. However, the main difficulty is that we have to work in a non-reflexive Banach space $L^1$, which does not possess weak compactness, so the classical theorem for maximal monotone operators in reflexive Banach space cannot be applied directly. In fact, {due to the loss of weak compactness} it is natural to allow a Radon measure being our solution $\frac{1}{u}$ and we do observe the singularity when $u$ approaches zero in numerical simulations \cite{hangjie}. Also see \cite{Xu2} for an example where a measure appears in the case of an exponential nonlinearity.
Therefore, we devote ourselves to the establishment of a general abstract framework for problems associated with nonlinear monotone operators in non-reflexive Banach spaces and to solve our problem \eqref{ueq} by the abstract framework.
Furthermore, the established abstract framework can be applied to a wide class of degenerate parabolic equations which can be recast as an abstract evolution equation with maximal monotone operator in some non-reflexive Banach space, for instance, to the degenerate exponential model studied in \cite{Xu2}.
The abstract framework is discussed precisely below.

\subsection{Formal observations and abstract setup}\label{sec1.2}
Denote by $ \varphi(h,t)$ as the step location when considered as a function of surface height $h$. Formally, we have
  $$u(h,t)=h_x( \varphi(h,t),t)=\frac{1}{\varphi_h(h,t)},$$
  and the $u$-equation \eqref{ueq} can be rewritten as $ \varphi$-equation
\begin{equation}\label{eq:phin}
     \varphi_{ t}=\Big(\frac{1}{ \varphi_h^3}\Big)_{hhh};
\end{equation}
{   for further details we refer to} the appendix of \cite{our2}.

Motivated by the $\varphi$-equation, we want to recast \eqref{ueq} as an abstract evolution equation.
 If $u$ has a positive lower-bound $u\geq \alpha>0,$ then \eqref{ueq} can be rewritten as
\begin{equation}\label{nueq}
\Big(\frac{1}{u}\Big)_t=(u^3)_{hhhh}, \quad u(0)=u_0.
\end{equation}
Formally, if we take $w_{hh}=\frac{1}{u}$, then we have
\begin{equation}\label{tmweq}
w_t=(w_{hh}^{-3})_{hh}.
\end{equation}
Since our problem \eqref{nueq} is in $1$-periodic setting, i.e., one period $[0,1)$, we also want $w$ to be periodic. {  Denote by $\mathbb{T}$ the $[0,1)$-torus. For measure space, we can define periodic distributions as distributions on $\mathbb{T}$, i.e., bounded linear functionals on $C^\8(\I)$. } Let the $\I$-periodic function $w$ be the solution of the Laplace equation
$$w_{hh}=\frac{1}{u}-c_0, \quad \int_\I w \d h=0,$$
with compatibility condition
$$\int_\I w_{hh} \d h =\int_\I \frac{1}{u}-c_0 \d h =0.$$
If  \eqref{nueq} holds a.e., then we have
$$\int_\I \frac{1}{u} \d h \equiv \int_\I \frac{1}{u_0} \d h =:c_0>0$$
due to the periodicity of $u$. However, we cannot show that \eqref{nueq} holds almost everywhere.
Actually, the possible existence of singular part for $w_{hh}$ or $\frac{1}{u}$ is intrinsic,
 since the equation \eqref{ueq} becomes degenerated when $u$ approaches zero. We cannot prevent $u$
 from touching zero, and can only show $w_{hh}=\frac{1}{u}-c_0\in \mathcal {M}(\I)$, where $\mathcal{M}$ is the set of finitely additive, finite, signed Radon measures. Hence the compatibility condition becomes
$$\int_\I\d \big(\frac{1}{u}-c_0\big)=0,$$
where $c_0$ is a positive constant. Moreover, we can illustrate the singularity in the following stationary solution.
Define a $\I$-periodic function $w(h)$ such that
 $$w(h)=\left\{\begin{array}{ll}
& -(h+\frac{1}{2})^2+\frac{1}{12}\quad \text{ for }h\in[-\frac{1}{2},0);\\
 &-(h-\frac{1}{2})^2+\frac{1}{12}\quad \text{ for }h\in[0,\frac{1}{2}).
 \end{array}\right.
 $$
 Then $w_{hh}=-2+2\delta_0$ where $\delta_0$ is the Dirac function at zero and $w$ is the stationary solution to \eqref{tmweq}.
It partially explains why we can not exclude the singular part for $w_{hh}$ or $\frac{1}{u}$.

Therefore,
in this paper we consider the parabolic evolution equation
\begin{equation}
w_t=[(w_{hh}+c_0)^{-3}]_{hh},\qquad w(0)=w^0,
\label{maineq}
\end{equation}
under the assumption $w$ is periodic with period $\I$ and has mean value zero in one period, i.e., $\int_\I w \d h =0$.

For $1\leq p < \8,\,k\in \mathbb{Z}$, {   set}
\begin{align*}
W^{k,p}_{\pero }(\I)&:=\{u\in W^{k,p}(\I); u(h)=u(h+1), \text{a.e. and $u$ has mean value zero in one period}\},
\\
L^{p}_{\pero }(\I)&:=\{u\in L^{p}(\I); u(h)=u(h+1), \text{a.e. and $u$ has mean value zero in one period}\}.
\end{align*}
Standard notations for Sobolev spaces are assumed above.
If $k<0$ and  $1\le p< +\8,\, 1<q\leq +\8$, $(1/p) + (1/q) = 1$, then it can be shown that $W^{k,q}_{\pero }(\I)$ is the dual of $W^{-k,p}_{\pero }(\I).$

Our main functional spaces will be
\begin{equation}
\label{space V}
V:=\{v\in W_{\pero }^{2,1}(\I)\},
\end{equation}
and
\begin{equation}
  \tdV:=\{u\in W^{1,2}_{\pero }(\I); u_{hh}\in \mathcal{M}(\I) \}.
\end{equation}
Define also
\begin{equation}
\label{space U}
U:=\{v\in L^2_{\pero }(\I)\}.
\end{equation}
Endow $U$ and $V$ with the norms $\|u\|_{U}:=\|u\|_{L^2(\I)}$ and $\|v\|_{V}:=\|v_{hh}\|_{L^1(\I)}$ respectively.
Note that the zero-mean conditions for functions of $V$ give the equivalence between
$\|\cdot\|_V$ and $\|\cdot\|_{W^{2,1}(\I)}$. Note also that
the embeddings $V\hookrightarrow U\hookrightarrow V'$ are all dense and continuous.

\bigskip
\noindent{\bf{The space $\tdV$.}}

Note also that any $\I$-periodic function $w$ who has mean value zero
such that $w_{hh}$ is a  finite Radon measure will
belong to $W^{1,2}_{\pero }(\I )$, since the first derivative $w_h$ is a BV function (the total variation
of $w_h$ is exactly the total mass of $w_{hh}$).
Thus we can endow the space $\tdV$ with the norm
\begin{equation*}
\|w\|_{W^{1,2}(\I )}+\|w_{hh}\|_{\mathcal{M}(\I)}.
\end{equation*}
Since $w$ is $1$-periodic and has mean value zero, we have
\begin{equation*}
\|w_h\|_{L^2(\I )} \le \|w_{hh}\|_{\mathcal{M}(\I)},\qquad
\|w\|_{L^2(\I )} \le \|w_{h}\|_{L^2(\I )}.
\end{equation*}
So we can use the equivalent norm
\begin{equation*}
\|w\|_{\tdV}:=\|w_{hh}\|_{\mathcal{M}(\I)} = \sup_{f\in C (\I),\ |f|\le 1 } \int_\I f\d w_{hh}.
\end{equation*}

The weak{  -*} convergence on $\tdV$ is then characterized as: a sequence $w^n$ converges weakly-* to $w$ in $\tdV$ if
$w^n$ converges weakly to $w$ in $W^{1,2}(\I )$, and
$w^n_{hh}$ converges weakly{  -*} to $w_{hh}$ in $\mathcal{M}(\I)$, i.e.
\begin{equation*}
\int_\I f \d w^n_{hh} \to \int_\I f \d w_{hh} \qquad \text{for any } f\in C (\I ) .
\end{equation*}

\bigskip
\noindent{\bf{Relations between $V$ and $\tdV$.} }

Since $V$ is not reflexive, we first present a characterization of the bidual space $V''$. For any $v\in V$,
we have $v_{hh}\in L^1_{\pero }(\I)$. Since also $C (\I)\sse L^\8(\I)$, we have:
\begin{enumerate}[(i)]
\item the dual space $V'=\{u\in (W^{2,1}_{\pero }(\I))'\}=W^{-2,\8}_{\pero }(\I)$ ;
\item for any $\xi\in V',$ $\eta\in V$, from the Riesz representation, there exists $\bar{\xi}\in L^\8(\I)$ such that
$$\langle \xi, \eta \rangle_{V',V}:=\int_\I \bar{\xi}\eta_{hh} \d h,$$
and we denote $\bar{\xi}_{hh}$ as $\xi$ without risk of confusion;
\item the bidual space $V''$ is a subspace of $\tdV$. Indeed, since $C (\I)\sse L^\8(\I)$, for any $u\in V''$ and any $g\in C(\I)$, we have 
$$|\lg u_{hh}, g\rg|=|\lg u, g_{hh}\rg_{(V'',V')}|\leq |u|_{V''}|g_{hh}|_{V'}\leq |u|_{V''}|g|_{C(\mathbb{T})}<+\8 ,$$
where we have used the identity
$$\langle g_{hh}, \eta \rangle_{(V', V)}=\int_{\mathbb{T}} g \eta_{hh} \mbox{d} h, \quad \forall \eta \in V$$
to conclude $|g_{hh}|_{V'}\leq |g|_{C(\mathbb{T})}.$
 Thus we know $u_{hh}$ define a bounded linear functional on $C(\I)$ so $u\in \tdV.$
\end{enumerate}
Thus
\begin{equation*}
V\sse V''\sse \tdV \sse W^{1,2}_{\pero }(\I)\sse U.
\end{equation*}
Therefore, we conclude that
the canonical embedding $V\hookrightarrow V''\hookrightarrow \tdV \hookrightarrow W^{1,2}_{\pero }(\I) \hookrightarrow U$ is continuous and each one is a dense subset of the next, since $V$ is dense in $U$.

\bigskip
\noindent{\bf{Observation 1.}}

 From \eqref{maineq}, one formal observation is that if we set
\begin{equation}
  \phi(w):=\frac{1}{2}\int_\I (w_{hh}+c_0)^{-2} \d h,
\end{equation}
then
$$w_t=-\frac{\delta \phi}{\delta w}=[(w_{hh}+c_0)^{-3}]_{hh} $$
forms a gradient flow of $\phi$ with the first variation $\frac{\delta \phi}{\delta w}$; see exact definition in \eqref{defphi} and calculations in Theorem \ref{mainth1}.
Hence we have
\begin{equation}
  \frac{\d \phi}{\d t}=\int_{\I} \frac{\delta \phi}{\delta w}w_t \d h=\int_{\I} -w_t^2 \d h =-\int_{\I} [(w_{hh}+c_0)^{-3}]_{hh}^2 \d h \leq 0.
\end{equation}
Besides, we also notice that  $\phi(w)=\frac{1}{2}\int_\I  (w_{hh}+c_0)^{-2} \d h$ is a convex functional. Recall that the sub-differential of a proper, convex, lower-semicontinuous function is a maximal monotone operator (see for instance \cite{Barbu2010}), which gives us the idea of using maximal monotone operator to formally rewrite our problem \eqref{maineq}, i.e.,
\begin{equation}
  w_t=-\partial \phi(w).
\end{equation}

\bigskip
\noindent{\bf{Observation 2.}}

Set also
\begin{equation}
  E(w):=\frac{1}{2}\int_{\I} [(w_{hh}+c_0)^{-3}]_{hh}^2 \d h =\frac{1}{2}\int_{\I} w_t^2 \d h;
\end{equation}
see exact definition in Definition \ref{strongslu}.
Taking {the derivative} $\partial_{hh}$ on the both side of \eqref{maineq}, we have
\begin{equation}\label{whh}
  [w_{hh}+c_0]_t=[(w_{hh}+c_0)^{-3}]_{hhhh}.
\end{equation}
 Then another formal observation is that
\begin{align*}
  \frac{\d E(w)}{\d t}&=\int_\I [(w_{hh}+c_0)^{-3}]_{hh}[(w_{hh}+c_0)^{-3}]_{hht} \d h\\
  &=\int_\I [(w_{hh}+c_0)^{-3}]_{hhhh}[(w_{hh}+c_0)^{-3}]_{t} \d h=\int_\I [w_{hh}+c_0]_{t}[(w_{hh}+c_0)^{-3}]_{t} \d h\\
  &=\int_\I -3 \frac{[(w_{hh}+c_0)_t]^2}{(w_{hh}+c_0)^4}\d h \leq 0.
\end{align*}
We point out the dissipation of $E(w)$ is important for the proof of existence result.

\bigskip
\noindent{\bf{Observation 3.}}

Moreover, to ensure the surjectivity of the maximal monotone operator $ \pd\phi$, we need to find a proper invariant ball. Another formal observation from \eqref{whh} is that
\begin{equation*}
  \frac{\d}{\d t} \int_\I (w_{hh}+c_0) \d h=\int_\I [(w_{hh}+c_0)^{-3}]_{hhhh} \d h =0.
\end{equation*}
So for a constant $C$ depending only on the initial value $w_0$, $\{\|w\|_{V}\leq C\}$ could be an invariant ball provided $w_{hh}+c_0>0$ almost everywhere. But note that $V$ is not a reflexive space and that bounded sets in $L^1{  (\I)}$
do not have any compactness property. Actually we only obtain
$$w_{hh}+c_0>0, \text{ a.e. }(t,h)\in[0,T]\times \I,$$
 $$\int_\I \d (w_{hh}+c_0) \leq C, \text{ for any }t\geq 0,$$ and
 choose $\{\|w\|_{\tdV }\leq C\}$ to be the invariant ball. That is consistent with the prediction that $w_{hh}=\frac{1}{u}-c_0$ is possible to be a Randon measure.

After those formal observations,
in order to rewrite our problem as an abstract problem precisely, we introduce the following definition.
\begin{defn}\label{defpp}
For any  $w\in \tdV ,$ from \cite[p.42]{evans1992}, we have the decomposition
\begin{equation}\label{decom}
  w_{hh}=\eta +\nu
\end{equation}
with respect to the Lebesgue measure,
where $\eta \in L^1(\I)$ is the absolutely continuous part of $w_{hh}$ and $\nu$ is the singular part, i.e., the support of $\nu$ has Lebesgue measure zero. Recall $c_0$ is a constant in \eqref{maineq}.
Denote $ g :=\eta +c_0$. Then $w_{hh}+c_0= g +\nu$ and $ g \in L^1(\I)$ is the absolutely continuous part of $w_{hh}+c_0.$

Define the proper, convex functional
\begin{equation}\label{defphi}
\phi:\tdV \lra \R\cup\{+\8\},\quad \phi(w):=
{
\begin{cases}
\int_{0}^1 \Phi( g ) \d h&{ \text{if } w_{hh}+c_0\in \mathcal{M}^+(\I),}\\
+\8&\text{otherwise},
\end{cases}}
\quad
\Phi(x):=\begin{cases}
+\8 & \text{if }x\le 0,\\
x^{-2}/2&\text{if }x> 0,
\end{cases}
\end{equation}
where $ g \in L^1(\I)$ is the absolutely continuous part of $w_{hh}+c_0.$
For some constant $C>0$ large enough, define the proper, convex functional
\begin{equation}\label{defpsi}
\psi:\tdV \lra \{0,+\8\},\qquad \psi(w):=\begin{cases}
0&\text{if } \|w\|_{ \tdV }\le C,\\
+\8&\text{if } \|w\|_{ \tdV }> C.
\end{cases}
\end{equation}
\end{defn}
The domain of $\phi+\psi$ is
$$D(\phi+\psi):=\{w\in\tdV;\, (\phi+\psi)(w)<+\8\}\sse \tdV \cap \{\|w\|_{\tdV }\le C\}.$$
Note that $\mathcal{K}:=\{w\in \tdV :\|w\|_{ \tdV }\le C\}$ is closed and convex, hence
its indicator (i.e., $\psi$) is convex, lower-semicontinuous and proper. Later, we will determine the constant $C$ by initial data $w_0$ and show $\psi$ is just an auxiliary functional.

Now we can state two definitions of solutions we study in this work.
\begin{defn}\label{weakslu}
 Given $\phi,\,\psi$ defined in Definition \ref{defpp}, for any $T>0,$ we call the function $$w\in L^\8([0,T];\tdV )\cap C^0([0,T];U),\quad w_t\in L^\8([0,T];U)$$
   a variational inequality solution to \eqref{maineq} if it satisfies
  \begin{equation}\label{vi}
 \lg w_t,v-w\rg_{{ U',U}} +(\phi+\psi)(v)-(\phi+\psi)(w)\ge 0
 \end{equation}
  for a.e. $t\in[0,T]$ and all $v\in { \tdV }$.
\end{defn}

\begin{defn}\label{strongslu}
  For any $T>0,$ let $\eta \in L^1(\I)$ be the absolutely continuous part of $w_{hh}$ in \eqref{decom}. Define
         \begin{equation}
           E(w):=\frac{1}{2}\int_\I \big[((\eta +c_0)^{-3})_{hh}\big]^2 \d h.
         \end{equation}
        We call the function $$w\in L^\8([0,T];\tdV )\cap C^0([0,T];U),\quad w_t\in L^\8([0,T];U)$$
   a strong solution to \eqref{maineq} if
   \begin{enumerate}[(i)]
  \item it satisfies
  \begin{equation}\label{slu}
  w_t=[(\eta +c_0)^{-3}]_{hh}
 \end{equation}
  for a.e. $(t,h)\in[0,T]\times \I$;
  \item we have $((\eta +c_0)^{-3})_{hh}\in L^\8([0,T];U)$ and the dissipation inequality
         \begin{equation}
           E(w(t))=\frac{1}{2}\int_\I \big[((\eta(t) +c_0)^{-3})_{hh}\big]^2 \d h\leq E(w(0)).
         \end{equation}
  \end{enumerate}
\end{defn}
The main result in this work is to prove
existence of the variational inequality solution and strong solution to \eqref{maineq}, which is stated in Theorem \ref{mainth} and Theorem \ref{mainth1} separately.

\subsection{Overview of our method and related method}
The key of our method is to rewrite the original problem as an abstract evolution equation $w_t=-\tilde{B}w$, where $\tilde{B}$
 is
the sub-differential of a proper, convex, lower semi-continuous function, i.e. $\tilde{B}=\pd (\phi+\psi)$. $\tilde{B}$ is a maximal monotone operator by classical results (see for instance \cite{Barbu2010}). $\psi$ is the indicator of the invariant ball $\mathcal{K}$ in \eqref{defpsi}. By constructing the proper invariant ball $\mathcal{K}$, we also obtain the restriction of $\tilde{B}$ to $L^2(\I)$ is also a maximal monotone operator; see Lemma \ref{dc}. Notice the {   definition of the} functional $\phi$ {   involves}
only the absolutely continuous part of $w_{hh}$, so we need to prove that it is still lower semi-continuous on $\tdV$; see details in Proposition \ref{newlsc}.  Then by standard theorem for m-accretive operator (see Definition \ref{set}) in \cite{Barbu2010}, we can prove the variational inequality solution to \eqref{maineq} in Theorem \ref{mainth}. Another key point is to prove the multi-valued operator $\pd (\phi+\psi)$ is actually single valued, which concludes
that the variational inequality solution is also the strong solution defined in Definition \ref{strongslu}. However, it is not easy to directly prove $\pd (\phi+\psi)$ is single valued, so we use Minty's trick to test the variational inequality \eqref{vi} with $v=w\pm \vep \vph$. After taking limit $\vep\to 0$, we can see $w_t+\pd\phi(w)$ is a zero function for a.e. $(t,x)\in[0,T]\times\I$; see details in Theorem \ref{mainth1}.

Actually, our definitions for variational inequality solution and strong solution in Definition{s} \ref{weakslu}
and \ref{strongslu} hide a Radon measure in it. As we said before, this kind of fourth order degenerate equation has the intrinsic property of singular measure. We want to mention that \cite{Leoni2015} also used maximal monotone operator method for diffusion limited (DL) case.
 However, {since the mobility for DL model is $M=1$ instead of $\frac{1}{h_x}$, DL model can be recast as an abstract evolution equation with maximal monotone operator using the anti-derivative of $h$.} The coercivity of the this maximal monotone operator in DL case is natural and hence the operated space is a reflexive Banach space. It is much easier than our case and singular part will not appear.

 Recently, \cite{our2,Xu1} also analyzed the positivity and the weak solution to the same equation \eqref{ueq} separately. They all considered this nonlinear fourth order parabolic equation, which comes from the same step flow model on vicinal surface. The aim is to answer the two questions in Section 1.1, which also are stated as open questions in \cite{Kohnbook}. The nonlinear structure of this equation, the key for both previous and current works, is important for the positivity of solution because it is known that
 the sign changing is a general property for solutions to linear fourth order parabolic equations. For one dimensional case, following the regularized method in \cite{Friedman1990}, \cite{our2} defined the weak solution on a subset, which has full measure, of $[0,1]$ and proved positivity and existence. Using the method of approximating solutions, based on
 the implicit time-discretization scheme and carefully chosen regularization, \cite{Xu1} expanded the result in \cite{our2} to higher dimensional case. Our results are consistent with theirs, but we use a totally different  approach.
The method adopted in  \cite{Xu1} is delicate and subtle while our method seems to be more general. Furthermore, we obtain the variational inequality solution to \eqref{maineq}. We also refer to \cite{Gradient} for deep study of gradient flow in metric space, in which the results can be stated in any Banach space including non-reflexive space since the purely metric formulation does not require any vector differentiability property. However they have almost no regularity result beyond Lipschitz regularity in space.

 We point out that our method establishes a general framework for this kind of equation whose invariant ball exists in a non-reflexive Banach space. We believe this method can be applied to many similar degenerated problems as long as they can be reduced to an abstract evolution equation with maximal monotone operator which  is unfortunately in a non-reflexive Banach space.

The rest of this work is devoted
to first recall some useful definitions in Section \ref{sec2}. Then in Section \ref{sec3}, we rigorously study the sub-differential $\pd(\phi+\psi)$ and prove it is m-accretive on $U$, which leads to the existence result for variational inequality solution. In Section \ref{sec4}, we calculate the exact value of $\pd(\phi+\psi)$ and prove the variational inequality solution is
 actually a strong solution.

\subsection{Preliminaries}\label{sec2}
In this section, we first recall the following classical definitions (see for instance \cite{Barbu2010}).
\begin{defn}\label{def4}
Given a Banach space $X$ with the duality pairing
$\langle,\rangle_{X',X}$, an element $x\in X$, a functional $f:X\lra \R\cup\{+\8\}$, the sub-differential
of $f$ at $x$ is the set defined as
\begin{equation*}
\pd f(x):=\{x'\in X':f(y)-f(x)\ge \langle x',y-x\rangle_{X',X}  \ \fal y\in X\}.
\end{equation*}
We denote the domain of $\pd f$ as usual by $D(\pd f)$, i.e. the set of all $x\in X$ such that $\pd f(x)\neq \emptyset.$
\end{defn}

\begin{defn}\label{maccr}
Given a Banach space $X$ with the duality pairing
$\langle,\rangle_{X',X}$, denote the elements of $X\times X'$ as $[x, y]$ where $x\in X, y\in X'$. A multivalued operator $A:X\lra X'$ identified with
its graph
$\ggm_A:=\{[x,y]\in X\times X';\,  y\in Ax \}\sse X\times X'$ is:
\begin{enumerate}
\item {\bf monotone} if for any pair $[u,u']$, $[v,v']\in \ggm_A$,
it holds
$$\lg u'-v',u-v\rg_{X',X} \geq 0;$$

\item {\bf maximal monotone} if the graph
$\ggm_A$ is not a proper subset of any monotone set.

\end{enumerate}

\begin{defn} \label{set}
Given a Hilbert space $X$, a multivalued
 operator $ B : X\lra X$ with graph $\ggm_{ B }:=\{[x,y]\in X\times X;\,  y\in Bx \}\sse X\times X$, denote $J_X:X\to X'$
as the canonical isomorphism of $X$ to $X'$. $B$ is
\begin{enumerate}
\item {\bf accretive} if for any pair $[u, u']$, $[v,v']\in \ggm_B$, there exists an element
$z\in J_X(u-v)$ such that $\lg z, u'-v'\rg_{X',X}\geq 0$;

\item {\bf m-accretive} if it is accretive and $R(I+ B )=X$, where $R(I+ B )$ denotes the range of $(I+B)$;

\end{enumerate}
\end{defn}
\end{defn}

 {\em{Remark:}} For general Banach space, $J_X$ is the duality mapping of $X$; see details in \cite[Section 1.1]{Barbu2010}. In our case, $X=U$, so $J_X$ is the identity
operator $I$ in $U$.

\bigskip

\section{Existence result for variational inequality solution}\label{sec3}

This section is devoted to obtain a variational inequality solution to \eqref{maineq}. {By restricting the operator in the  non-reflexive Banach space $\tdV$ to $U$,} we want to apply the classical result for m-accretive operator in $U$. However, since we do not have weak compactness for sequences in $V$, and a Radon measure
{may appear} when taking the limit, we need to first prove {weak-*} lower semi-continuity for functional $\phi$ in $\tdV$.

\subsection{{  Weak-*} lower semi-continuity for functional $\phi$ in $\tdV$.}

{
Since for any $w\in\tdV$, $\phi$ defined only on its absolutely continuous part, we need the following proposition to guarantee $\phi$ is lower-semi-continuous with respect to the weakly-* convergence in $\tdV$.
\begin{prop}\label{newlsc}
  The function $\phi$ defined in Definition \ref{defpp} is lower semi-continuous with respect to the weakly-$*$ convergence in $\tdV$, i.e.,
  if $w_n {\overset*\rhu} w$ in $\tdV$, we have
  $$\liminf_{n\to +\8}\phi(w_n)\geq \phi(w).$$
\end{prop}
For any $\mu\in \mathcal{M}(\I)$, we denote $\mu\ll\L^1$ if $\mu$ is absolutely continuous with respect to Lebesgue measure and denote $\bar{\mu}:=\frac{\d\mu}{\d\L^1}$ as the density of $\mu$. For notational simplification, denote $\mu_{\|}$ (resp. $\mu_{\bot}$) as the absolutely continuous part (resp. singular part) of $\mu$ with respect to Lebesgue measure.
Before proving Proposition \ref{newlsc}, we first state some lemmas. {  The following Lemma comes from the weak-* compactness of $L^\8$ directly so we omit the proof here.
\begin{lem}\label{lemdunford}
  For any $N\geq0$, given a sequence of measures $\mu_n$ in $\mathcal{M}(\I)$ such that  $\mu_n\ll\L^1$ for any $n$, and the densities $\bar{\mu}_n:=\frac{\d\mu_n}{\d\L^1}$ satisfy
\begin{equation*}
\sup_n\big\|\bar{\mu}_n\big\|_{L^\infty(\I)} \leq N,
\end{equation*}
then there exist a measure $\mu\ll\L^1,\,\|\bar{\mu}\|_{L^{\8}(\I)}\leq N$ and a subsequence $\mu_{n_k}$ such that $\mu_{n_k}\overset*\rhu \mu$ in $\mathcal{M}(\I)$.
\end{lem}
}

From now on, we identify $\mu_n$ with its density $\bar{\mu}_n:=\frac{\d\mu_n}{\d\L^1}$ and do not distinguish them for brevity.
Given a sequence of measures $\mu_n$ such that $\mu_n\ll\L^1$, $\bar{\mu}_n\ge 0$ and $N>0$,
 observe that
\begin{equation}\label{b}
\mu_n = \min\{\mu_n,N\}+\max\{\mu_n,N\}-N.
\end{equation}
From Lemma \ref{lemdunford} we know, upon subsequence, $\min\{\mu_n,N\} \overset*\rhu\mu_-$ for some measure
$\mu_-$ satisfying $\mu_-\ll\L^1$ and $N\ge \mu_-\ge 0$. We also need the following useful Lemma to clarify the relation between $\mu_-$ and the weak-$*$ limit of $\mu_n$.
\begin{lem}\label{lem0105}
 Given a sequence of measures $\mu_n$ such that $\mu_n\ll\L^1$ in $\mathcal{M}(\I)$, ${\mu}_n\ge 0$, we assume moreover that $\mu_n\overset*\rhu\mu$, for some measure $\mu\ge 0$. Then for any $N>0$, there exist $\mu_-,\,\mu_+\in \mathcal{M}(\I),$ such that
 \begin{equation}\label{starlem1}
   \min\{\mu_n,N\}\overset*\rhu \mu_- \text{ in }\mathcal{M}(\I),\quad \mu_-\ll\L^1,\quad \mu_-\leq \mu_{\|},
 \end{equation}
  \begin{equation}\label{starlem2}
   \max\{\mu_n,N\}\overset*\rhu \mu_{+} \text{ in }\mathcal{M}(\I),\quad \mu_{+\|}\ge N,
 \end{equation}
 where $\mu_{\|}$ (resp. $\mu_{\bot}$) is the absolutely continuous part (resp. singular part) of $\mu$. Moreover, for the function $\Phi$ defined in \eqref{defphi}, we have
 \begin{equation}\label{PHI}
   \Phi(\mu_\|)\leq \Phi(\mu_-).
 \end{equation}
\end{lem}
\begin{proof}
From Lemma \ref{lemdunford} we know, upon subsequence, $\min\{\mu_n,N\} \overset*\rhu\mu_-$ for some measure
$\mu_-$ satisfying $\mu_-\ll\L^1$ and $N\ge \mu_-\ge 0$. By Lebesgue decomposition theorem, there exist unique measures $\mu_\|\ll\L^1$ and $\mu_\bot\bot\L^1$
such that $\mu=\mu_\|+\mu_\bot$.
The decomposition \eqref{b} then gives
$$0\le \mu_n-\min\{\mu_n,N\} = \max\{\mu_n,N\}-N \overset*\rhu\mu - \mu_-.$$
 Taking $\mu_+:=\mu - \mu_-+N$, since the sequence $\max\{\mu_n,N\}-N\ge0$, we know $\max\{\mu_n,N\}\overset*\rhu \mu_{+}$ and $(\mu - \mu_-)_\|=\mu_{+\|}-N\ge 0$. Besides,  since $\Phi(\mu_{\|})$ is decreasing with respect to $\mu_{\|}$, we obtain \eqref{PHI}.
\qed\end{proof}


Now we can start to prove Proposition \ref{newlsc}.

\begin{proof}[Proof of Proposition \ref{newlsc}]
Without loss of generality we may assume $\sup_{n\to+\8} \phi(w_n)<+\8$.
{  This immediately implies  that all $(w_n)_{hh}+c_0$ are positive measures.
Assume $w_n {\overset*\rhu} w$ in $\tdV$, thus we have $(w_n)_{hh}{\overset*\rhu} w_{hh}$ in $\mathcal{M}(\I).$ Denote $f_n:=(w_n)_{hh}+c_0$ and $f:=w_{hh}+c_0$.} Since $\Phi(f_{\|})$ is decreasing with respect to $f_{\|}$, we only concern the case $f_{n\|}$ may weakly-* converge to a singular measure. Thus without loss of generality, we may assume $f_n\ll\L^1$, i.e., $f_{n\bot}=0$.  For any $M>0$ large enough,
denote $\phi_{M}(w_n):=\int_\I  \Phi(\min\{f_n,M\}).$
From the definition of $\Phi$ in \eqref{defphi},
the truncated measures $\min\{f_n,M\}$ satisfy
\begin{align*}
 \phi_M(w_n) &=\int_\I  \Phi(\min\{f_n,M\})\d h\\
 & = \int_{\{f_n\le M\}}\Phi(\min\{f_n,M\})\d h+\frac{1}{2M^2}\L^1(\{f_n> M\})\\
&\ge\int_{\{f_n\le M\}}\Phi(f_n)\d h+\int_{\{f_n> M\}}\Phi(f_n)\d h= \phi(w_n).
\end{align*}
The second equality also shows
\begin{align*}
  \phi_M(w_n)-\frac{1}{2M^2}\L^1(\{f_n> M\})&=\int_{\{f_n\le M\}}\Phi(\min\{f_n,M\})\d h\\
  &=\int_{\{f_n\le M\}}\Phi(f_n)\d h\\
  &\leq \int_\I  \Phi(f_n) \d h =\phi(w_n).
\end{align*}
Hence we obtain
\begin{equation}\label{0104_28o}
  |\phi(w_n)-\phi_M(w_n)|\leq \frac{1}{2M^2}\L^1(\{f_n> M\})\leq \frac{1}{2M^2}.
\end{equation}

From Lemma \ref{lemdunford} and Lemma \ref{lem0105}, we know the truncated sequence $\min\{f_n,M\}$ satisfies
\begin{equation}\label{star0105_1}
   \min\{f_n,M\}\overset*\rhu f_- \text{ in }\mathcal{M}(\I),\quad f_-\ll\L^1,\quad
   \Phi(f_\|)\leq \Phi(f_-).
 \end{equation}
Hence by the convexity and lower semi-continuity of $\phi$ on $V$, we infer
\begin{equation}\label{conv}
\liminf_{n\to+\8}\int_\I  \Phi(\min\{f_n,M\})\d h\ge \int_\I  \Phi(f_-)\d h\geq \int_\I  \Phi(f_\|)\d h=\phi(w).
\end{equation}
Combining this with \eqref{0104_28o}, we obtain
\begin{equation}
\begin{aligned}
  \liminf_{n\to+\8}\phi(w_n)&\geq \liminf_{n\to+\8} \phi_M(w_n)-\frac{1}{2M^2}\\
  &=\liminf_{n\to+\8}\int_\I  \Phi(\min\{f_n,M\})\d h-\frac{1}{2M^2}\\
  &\geq \phi(w)-\frac{1}{2M^2},
  \end{aligned}
\end{equation}
and thus we complete the proof of Proposition \ref{newlsc} by the arbitrariness of $M$.
\qed\end{proof}
}

\subsection{Maximal monotone and m-accretive operator in $U$.}
In this section, we first define the sub-differential of $\phi+\psi$ and then obtain a useful lemma to ensure $\pd (\phi+\psi)$ is also a maximal monotone operator when restricted to $U$.

 Let $\tilde{B}:= \pd(\phi+\psi):\tdV \to \tdV '$ be
 the sub-differential of $\phi+\psi.$ Let us consider the operator $B$ as the restriction of $\tilde{B}$ from $U$ to $U'$.
\begin{defn}\label{defB}
Define the operator $B: D(B)\sse U\to U'$ such that
\begin{equation*}
  Bw=\tilde{B}w, \text{ for any }w\in D(B)=\{w\in \tdV ;\tilde{B}w\sse U\}.
\end{equation*}
\end{defn}

We first prove $B$ is maximal monotone in $U\times U'$, which is important to prove the existence result.

\begin{lem}\label{dc}
The operator ${ B:D(B)\sse U\to U'}$ in Definition \ref{defB} is maximal monotone in $U\times U'$.
\end{lem}

\begin{proof}
It suffices to prove that $\phi+\psi$ is (i) proper, i.e. $D(\phi+\psi)\neq \emptyset$, (ii) convex and (iii) lower semi-continuous
when considered as a functional from $(U,\|\cdot\|_U)$ to $\R\cup \{+\8\}$.
(i) First it is clear that $\phi+\psi$ is proper.

\medskip

{\em (ii) Convexity.} Let $u_1,u_2\in U$ be arbitrarily given, and
we need to show
\begin{equation*}
(1-t)(\phi+\psi)(u_1)+t(\phi+\psi)(u_2)\ge (\phi+\psi)((1-t)u_1+tu_2).
\end{equation*}
If either $u_1$ or $u_2$ does not belong to $D(\phi+\psi)$,
then the left-hand side term is $+\8$.
If both $u_1$ and $u_2$ belong to $D(\phi+\psi)$,
then $(1-t)u_1+tu_2$ also belongs to ${ \tdV }\cap\{\|\cdot\|_{ \tdV }\le C\}$, hence
\begin{equation*}
\psi(u_1)=\psi(u_2)=\psi((1-t)u_1+tu_2)=0.
\end{equation*}
Notice the convexity of $\phi$, and the fact that the absolutely continuous part of
$((1-t)u_{1}+tu_2)_{hh}$ is $(1-t)((u_1)_{hh})_{\|}+t((u_2)_{hh})_{\|}$, where $((u_i)_{hh})_{\|}$ are notations representing the absolutely continuous parts of $(u_i)_{hh}$ separately. Then we obtain
\begin{equation*}
(1-t)\phi(u_1)+t\phi(u_2)\ge \phi((1-t)u_1+tu_2).
\end{equation*}
Thus $\phi+\psi$ is convex.

\medskip

{\em (iii) Lower-semicontinuity.} Note that
the lower-semicontinuity is here intended as with respect to the strong
convergence in $U$ (less restrictive than the convergence in ${\tdV }$). Consider an arbitrary sequence $u^n\sse U$ converging to $u\in U$. We need
to prove
\begin{equation*}
(\phi+\psi)(u)\le \liminf_{n\to+\8}(\phi+\psi)(u^n).
\end{equation*}
If $\liminf_{n\to+\8}(\phi+\psi)(u^n)=+\8$ then the thesis is trivial. Thus assume (upon subsequence)
\begin{equation*}
\liminf_{n\to+\8}(\phi+\psi)(u^n)=\lim_{n\to+\8}(\phi+\psi)(u^n)\le D<+\8.
\end{equation*}
Without loss of generality we can further assume $(u^n)_n\sse { \tdV }\cap\{\|\cdot\|_{ \tdV }\le C\}$.
This implies $\|u_{hh}^n\|_{\mathcal{M}(\I)}\le C$,
so
\begin{equation*}
\qquad \psi(u^n)=0,\quad \fal n
\end{equation*}
and there exists $\xi\in \mathcal{M}(\I)$ such that $u_{hh}^n{\overset*\rhu} \xi:=v_{hh}$ in
$\mathcal{M}(\I)$. Since from Proposition \ref{newlsc}, $\phi+\psi$ is convex and weak-* lower-semicontinuous in $\tdV $,
we infer
\begin{equation}
{ \phi}(v)\le \liminf_{n\to+\8} \phi(u^n).\label{f lsc}
\end{equation}
The uniform boundedness of $\|u_{hh}^n\|_{\mathcal{M}(\I)}$ also implies
 $(u^n)_n $ is bounded in $W^{1,\8}(\I)$ (and hence in
$W^{1,p}(\I )$ for any $p<+\8$). Thus $(u^n)_n$ is (upon subsequence) weakly convergent
in $W^{1,p}(\I )$, and strongly convergent in $L^2(\I )$ to $u$. Thus $u_{hh}=\xi=v_{hh}$,
and \eqref{f lsc} now becomes
\begin{equation*}
\phi(u)={ \phi}(v)\le \liminf_{n\to+\8} \phi(u^n).
\end{equation*}

\medskip

Therefore $\phi+\psi:U\lra \R\cup\{+\8\}$ is proper,
convex and lower-semicontinuous. Then by \cite[Theorem 2.8]{Barbu2010} we have that $\pd(\phi+\psi):U\lra U'$ is maximal monotone in $U\times U'$.
\qed\end{proof}

%

Notice $U'=U$. From Lemma \ref{dc} and the Definition \ref{set} we  deduce
\begin{prop}\label{lem9}
The operator { $B:D(B)\sse U\to U'$} in Definition \ref{defB}
is m-accretive from $D(B)\sse U$ to $U$.
\end{prop}

\medskip

\subsection{Existence of variational inequality solution.}

After those preliminary results, we can apply \cite[Theorem 4.5]{Barbu2010} to obtain the existence of variational inequality solution to \eqref{maineq}.

First let us recall \cite[Theorem 4.5]{Barbu2010}.
\begin{thm}(\cite[Theorem 4.5]{Barbu2010})\label{barbuth}
For any $T>0$, let $U$ be a Hilbert space and let $B$ be a m-accretive operator from $D(B)\sse U$ to $U$. Then for each $y_0\in D(B)$, the cauchy problem
\begin{equation*}
  \left\{
     \begin{array}{ll}
       \frac{\d y}{\d t}(t)+By(t)\ni 0, \quad t\in[0,T], \\
       y(0)=y_0,
     \end{array}
   \right.
\end{equation*}
has a unique strong solution $y\in W^{1,\8}([0,T];U)$ in the sense  that
$$-\frac{\d y}{\d t}(t) \in By(t), \quad a.e.\, t\in[0,T],\qquad y(0)=y_0. $$
Moreover, $y$ satisfies the estimate
\begin{equation}\label{es1107_1}
  \|y_t\|_U\leq |-By_0|_\star,
\end{equation}
where $|-By_0|_\star=\inf\{\|u\|_U;u\in -By_0\}.$
\end{thm}

 Proposition \ref{lem9} shows that $B$ defined in Definition \ref{defB} is m-accretive from $D(B)\sse U$ to $U$. Hence we can apply Theorem \ref{barbuth} to obtain
\begin{thm}\label{mainth}
Let $B: D(B)\sse U\to U$  be the operator defined in Definition \ref{defB}. Given $T>0$, initial datum $w^0\in D(B)$, then
\begin{enumerate}[(i)]
  \item there exists a unique function $w\in W^{1,\8}([0,T];U) $ such that
        \begin{equation}\label{slu1}
        -w_t(t)\in Bw(t) ,\qquad w(0)=w^0, \quad\text{for } a.e.\, t\in[0,T].
        \end{equation}
  \item $w$ is also
        a variational inequality solution to \eqref{maineq}. {  Moreover,
        \begin{equation}\label{positive}
        \begin{aligned}
        &w_{hh}+c_0\in \mathcal{M}^+(\I),\text{ for a.e. }(t,h)\in[0,T]\times\I,\\
        &\eta +c_0>0, \text{ for a.e. }(t,h)\in[0,T]\times\I,
        \end{aligned}
        \end{equation}
         where a.e. means with respect to the Lebesgue measure, $\eta $ is the absolutely continuous part of $w_{hh}$ in \eqref{decom}, and $\mathcal{M}^+(\I)$ denotes the set of positive {Radon} measures.}
\end{enumerate}

\end{thm}
\begin{proof}
  {{Proof of (i).}} From Proposition \ref{lem9}, we know $B$ is a m-accretive operator in $U\times U$. So \eqref{slu1}
  {follows from}
  Theorem \ref{barbuth} and we have $w\in W^{1,\8}([0,T];U) $. From \eqref{es1107_1}, we also have
  \begin{equation}\label{es1107_2}
    \|w_t\|_U\leq |-Bw_0|_\star,
  \end{equation}
where $|-Bw_0|_\star=\inf\{\|u\|_U;u\in -Bw_0\}.$

{{Proof of (ii).}} Since
$-w_t(t)\in Bw(t)=\pd(\phi+\psi)( w)$ for $a.e.\, t\in[0,T]$ and $w\in W^{1,\8}([0,T];U) $, we see from Definition \ref{def4} that
\begin{equation}\label{vi2}
 \lg w_t,v-w\rg_{{ U',U}} +(\phi+\psi)(v)-(\phi+\psi)(w)\ge 0, \quad \text{a.e. } \,t\in[0,T]
 \end{equation}
 for all $v\in U$,
and
$$w\in C^0([0,T];U),\quad w_t\in L^\8([0,T];U).$$
Choose a function $v\in U$ such that $(\phi+\psi)(v)\leq 1$. Then from \eqref{vi2}, we also have
\begin{equation}\label{45}
  (\phi+\psi)(w)\leq \|w_t\|_U\|w-v\|_U+1 \quad\text{  for a.e. }t\in[0,T].
\end{equation}
This implies
$$w\in L^\8([0,T];\tdV )$$
and {with respect to the Lebesgue measure,
$$  \eta +c_0>0\,\,  \text{ for a.e. }(t,h)\in[0,T]\times\I,$$
$$w_{hh}+c_0\in \mathcal{M}^+(\I)\,\, \text{ for a.e. }(t,h)\in[0,T]\times\I,$$
where $\eta $ is the absolutely continuous part of $w_{hh}$ in \eqref{decom}.}
Therefore we obtain the variational inequality solution to \eqref{maineq} and $w$
{satisfies} the positivity {property} \eqref{positive}.
\qed\end{proof}

\section{Existence of strong solution}\label{sec4}
Although we obtained a unique variational inequality solution in Theorem \ref{mainth}, we do not know whether $B$ is single-valued and which element belongs to $B$. We will prove the variational inequality solution is actually a strong solution in this section.

Now we assume $$w\in L^\8([0,T];\tdV )\cap C^0([0,T];U),\quad w_t\in L^\8([0,T];U)$$
 is the  variational inequality solution to \eqref{maineq}, i.e., $w$ satisfies
  \begin{equation}\label{8vi}
 \lg w_t(t),v-w(t)\rg_{{ U',U}} +(\phi+\psi)(v)-(\phi+\psi)(w(t))\ge 0
 \end{equation}
  for a.e. $t\in[0,T]$ and all $v\in { \tdV }$.

Let $\vph\in C^{\8}(\I)$ be given. The idea is to test \eqref{8vi} with $v:=w\pm\vep \vph$.
However, in general this is not possible, since it is not guaranteed that $v=w\pm\vep \vph\in D(\phi+\psi)$. To handle this difficulty, we will use the truncation method in \cite{Leoni2015} to truncate $w_{hh\|}$ from below such that $v=w\pm\vep \vph\in D(\phi+\psi)$ for small $\vep$.
Let us state existence result for strong solution as follows.

\begin{thm}\label{mainth1}
  Given $T>0$, initial datum $w^0\in D(B)$, then the variational inequality solution $w$ obtained in Theorem \ref{mainth} is  also a strong solution
  to \eqref{maineq}, i.e.,
        \begin{equation}\label{main21_3}
         w_t=((\eta +c_0)^{-3})_{hh}
        \end{equation}
        for a.e. $(t,h)\in[0,T]\times \I$.
        Besides, we have $((\eta +c_0)^{-3})_{hh}\in L^\8([0,T];U)$ and the dissipation inequality
         \begin{equation}\label{dissiE}
           E(t):=\frac{1}{2}\int_\I \big[((\eta +c_0)^{-3})_{hh}\big]^2 \d h\leq E(0),
         \end{equation}
        where $\eta $ is the absolutely continuous part of $w_{hh}$ in \eqref{decom}.
\end{thm}
\begin{proof}

{\bf Step 1. Truncate $w_{hh\|}$ from below.}

Assume $w$ is the variational inequality solution to \eqref{maineq}. Choose an arbitrary $t$ for which the variational inequality \eqref{8vi} holds.
Let $\vph\in C^{\8}(\I)$ be given. Denote by $w_{hh\|}(t)$ (resp. $w_{hh\bot}(t)$) the absolutely continuous part (resp. singular part)
of $w_{hh}(t)$. In the following, we truncate $w_{hh}(t)$ below.
Let
\begin{equation}\label{trun01}
w^\dt_{hh}(t):= w_{hh}(t)+\dt\mathbf{1}_{E_\dt},\qquad E_\dt:=\{w_{hh\|}(t)\le \dt-c_0\}.
\end{equation}
We remark here a constant $-\delta |E_\delta|$ should be added to ensure the periodic setting, however we omit it for simplicity since the proof is same.
Since $w_{hh\|}(t)+c_0>0$ a.e., we can see
 \begin{equation}\label{tm911_1}
 |E_\dt|\to 0 \text{ as } \dt\to0.
 \end{equation}
Let
\begin{equation}
v:=w^\dt(t)+\vep \vph,\qquad \vep:=\frac{\dt}{2\|\vph\|_{W^{2,\8}(\I)}+1}.
\label{v}
\end{equation}
Now we prove $v=w^\dt(t)+\vep \vph\in D(\phi+\psi)$.
Note that
\begin{equation}\label{vphr}
|\vep \vph_{hh}|\leq \vep\|\vph\|_{W^{2,\8}(\I)}\leq \frac{\dt}{2},
\end{equation}
due to \eqref{v}.
First from
\begin{equation*}
v_{hh\|}+c_0=w_{hh\|}^\dt(t)+\vep \vph_{hh}+c_0 \ge  \dt - \vep \|\vph_{hh}\|_{L^\8(\I)} \ge \dt/2
\end{equation*}
we know
$v\in D(\phi).$
  Second from $w_{hh}+c_0\in \mathcal{M}^+$,  we know
 \begin{align*}
 \int_\I |w_{hh}| \d h &\leq \int_\I |w_{hh}+c_0|+|c_0| \d h\\
 &\leq \int_\I w_{hh}+c_0 \d h+\int_\I |c_0|\d h\\
 &= c_0+|c_0|=2c_0
\end{align*}
due to $c_0$ is positive.
 Hence we can choose $C:=2c_0+1$ in Definition \eqref{defpsi} to ensure $\|w\|_{L^\8(0,T;\~V)}\leq C-1$.
Then by construction, $v$ satisfies
\begin{align*}
\|v_{hh}\|_{\mathcal{M}(\I)} &=\|w_{hh}(t)+\dt\mathbf{1}_{E_\dt}+\vep\vph_{hh}\|_{\mathcal{M}(\I)}\\
&\le \|w_{hh}(t)\|_{\mathcal{M}(\I)} + \dt|E_\dt| + \vep\|\vph_{hh}\|_{\mathcal{M}(\I)}
\le C-1+\dt|E_\dt|+\dt/2,
\end{align*}
which implies
\begin{equation*}
\|v\|_{\~V}=\|v_{hh}\|_{\mathcal{M}(\I)}\le C
\end{equation*}
and $v\in D(\psi)$  for all sufficiently small  $\dt.$

\bigskip

{\bf Step 2. Integrability results. }

We claim
\begin{equation}\label{L1}
(w_{hh\|}(t)+c_0)^{-3}\in L^1(\I)
\end{equation}
\begin{equation}
\label{well-defined}
(w_{hh\|}^\dt(t)+\vep\vph_{hh}+c_0)^{-3}\in L^1(\I),
\end{equation}
for all sufficiently small $\dt$.

{\em Proof of \eqref{L1}}. First, for all $0<\vep\ll1$ we have
$$\|(1-\vep)w(t)\|_{\~V}=(1-\vep)\|w(t)\|_{\~V}$$
which implies $(1-\vep)w(t)\in D(\psi).$
Moreover, on $\{w_{hh\|}(t)\ge 0\}$ we have $(1-\vep)w_{hh\|}(t)+c_0\ge c_0>0$,
while on $\{w_{hh\|}(t)\le 0\}$ we have $(1-\vep)w_{hh\|}(t)\ge w_{hh\|}(t)$. Hence
$(1-\vep)w_{hh\|}(t)+c_0\ge w_{hh\|}(t)+c_0>0$ a.e., and
\begin{align*}
\int_{\I} ((1-\vep)w_{hh\|}(t)+c_0)^{-2}\d h& = \int_{\{w_{hh\|}(t)\ge0\}} ((1-\vep)w_{hh\|}(t)+c_0)^{-2}\d h\\
&+
\int_{\{w_{hh\|}(t)<0\}} ((1-\vep)w_{hh\|}(t)+c_0)^{-2}\d h\\
& \le \int_{\{w_{hh\|}(t)\ge0\}} c_0^{-2}\d h +\int_{\{w_{hh\|}(t)<0\}} (w_{hh\|}(t)+c_0)^{-2}\d h<+\8.
\end{align*}
Thus we have $(1-\vep)w(t)\in D(\phi)$.

Next,
setting $v=(1-\vep)w(t)\in D(\phi+\psi)$ in \eqref{8vi}, we get
\begin{equation}
\lg w_t(t),-\vep w(t)\rg+\phi((1-\vep)w(t))-\phi(w(t))\ge 0.\label{w vi}
\end{equation}
Direct computation gives
\begin{align*}
\phi((1-\vep)w(t))-\phi(w(t)) &= \frac12\int_\I [((1-\vep)w_{hh\|}(t)+c_0)^{-2}-(w_{hh\|}(t)+c_0)^{-2}]\d h\\
&=\frac12\int_\I \frac{(w_{hh\|}(t)+c_0)^{2}-((1-\vep)w_{hh\|}(t)+c_0)^{2}}{((1-\vep)w_{hh\|}(t)+c_0)^{2}(w_{hh\|}(t)+c_0)^{2}}\d h\\
&=\vep\int_\I \frac{w_{hh\|}(t)}{((1-\vep)w_{hh\|}(t)+c_0)^{2}(w_{hh\|}(t)+c_0)}\d h\\
&\qquad-\frac{\vep^2}{2}\int_\I \frac{|w_{hh\|}(t)|^2}{((1-\vep)w_{hh\|}(t)+c_0)^{2}(w_{hh\|}(t)+c_0)^2}\d h.
\end{align*}
Hence \eqref{w vi} gives
\begin{align*}
0&\le
\lg w_t(t),-\vep w(t)\rg + \vep\int_\I \frac{w_{hh\|}(t)}{((1-\vep)w_{hh\|}(t)+c_0)^{2}(w_{hh\|}(t)+c_0)}\d h\\
&=\lg w_t(t),-\vep w(t)\rg +\vep\int_\I \frac{1}{((1-\vep)w_{hh\|}(t)+c_0)^{2}}\d h
-\vep\int_\I \frac{c_0}{((1-\vep)w_{hh\|}(t)+c_0)^{2}(w_{hh\|}(t)+c_0)}\d h.
\end{align*}
This, together with $|\lg w_t(t), w(t)\rg|\le \|w_t(t)\|_U\|w(t)\|_U$, shows that
\begin{equation}\label{tm0906_1}
-\8<\lg w_t(t), w(t)\rg \le \int_\I \frac{1}{((1-\vep)w_{hh\|}(t)+c_0)^{2}}\d h -
\int_\I \frac{c_0}{((1-\vep)w_{hh\|}(t)+c_0)^{2}(w_{hh\|}(t)+c_0)}\d h
\end{equation}
for all $0<\vep\ll1$.
For the first term on the right hand side of \eqref{tm0906_1}, note
\begin{align*}
\frac{1}{((1-\vep)w_{hh\|}(t)+c_0)^{2}} &\to \frac{1}{(w_{hh\|}(t)+c_0)^{2}} \qquad \text{for a.e. }h,\\
 \frac{1}{((1-\vep)w_{hh\|}(t)+c_0)^{2}}  &\le c_0^{-2} \qquad \text{on } \{w_{hh\|}(t)\ge 0\},\\
 \frac{1}{((1-\vep)w_{hh\|}(t)+c_0)^{2}} &\le
 \frac{1}{(w_{hh\|}(t)+c_0)^{2}}\qquad \text{on } \{w_{hh\|}(t)< 0\},
\end{align*}
where $(w_{hh\|}(t)+c_0)^{-2}\in L^1(\I)$ due to $w\in D(\phi)$. Thus by Lebesgue's dominated convergence theorem
we have
\begin{equation}\label{tm_a}
\int_\I \frac{1}{((1-\vep)w_{hh\|}(t)+c_0)^{2}}\d h \to \int_\I \frac{1}{(w_{hh\|}(t)+c_0)^{2}}\d h=2\phi(w(t)).
\end{equation}
For the second term on the right hand side of \eqref{tm0906_1},
notice that on $\{w_{hh\|}(t)\ge 0\}$ we have
$$\frac{1}{((1-\vep)w_{hh\|}(t)+c_0)^{2}(w_{hh\|}(t)+c_0)}  \le c_0^{-3},$$
which implies
\begin{equation}
\label{tm_b}
 \int_{\{w_{hh\|}(t)\ge 0\}} \frac{1}{((1-\vep)w_{hh\|}(t)+c_0)^{2}(w_{hh\|}(t)+c_0)}\d h \to  \int_{\{w_{hh\|}(t)\ge 0\}} \frac{1}{(w_{hh\|}(t)+c_0)^{3}}\d h
 \end{equation}
due to Lebesgue's dominated convergence theorem.
On the other hand, on $\{w_{hh\|}(t)< 0\}$
\begin{align*}
\frac{1}{((1-\vep)w_{hh\|}(t)+c_0)^{2}(w_{hh\|}(t)+c_0)}
\end{align*}
is increasing with respect to $\vep$. Hence by the monotone convergence theorem we have
\begin{equation}
\label{tm_c}
\int_{\{w_{hh\|}(t)<0\}} \frac{1}{((1-\vep)w_{hh\|}(t)+c_0)^{2}(w_{hh\|}(t)+c_0)}\d h \to  \int_{\{w_{hh\|}(t)< 0\}} \frac{1}{(w_{hh\|}(t)+c_0)^{3}}\d h.
\end{equation}
Combining \eqref{tm_a}, \eqref{tm_b} and \eqref{tm_c}, we can take $\vep\to 0$ in \eqref{tm0906_1} to see that
\begin{align*}
-\8&<\lg w_t(t), w(t)\rg - 2\phi(w(t))\\
& \le -
c_0\int_\I \frac{1}{(w_{hh\|}(t)+c_0)^{2}(w_{hh\|}(t)+c_0)}\d h,
\end{align*}
which completes the proof of \eqref{L1}.

\medskip

{\em Proof of \eqref{well-defined}}.
Note that
\begin{equation*}
\int_\I (w_{hh\|}^\dt(t)+\vep\vph_{hh}+c_0)^{-3}\d h=\int_{\I\backslash E_\dt} (w_{hh\|}^\dt(t)+\vep\vph_{hh}+c_0)^{-3}\d h
+\int_{E_\dt} (w_{hh\|}^\dt(t)+\vep\vph_{hh}+c_0)^{-3}\d h.
\end{equation*}
First, on $\I\backslash E_\dt=\{w_{hh\|}(t) +c_0\ge \dt  \}$ we have $w_{hh\|}^\dt(t)=w_{hh\|}(t)$. Thus from \eqref{vphr} we have
\begin{align}
w_{hh\|}^\dt(t)+\vep\vph_{hh}+c_0 &=w_{hh\|}(t)+\vep\vph_{hh}+c_0 \ge w_{hh\|}(t)+c_0 -\dt/2
\ge(w_{hh\|}(t)+c_0)/2\label{upper estimate}
\end{align}
on $\I\backslash E_\dt$ and
\begin{equation}\label{well-defined-1}
\int_{\I\backslash E_\dt} (w_{hh\|}^\dt(t) +\vep\vph_{hh}+c_0)^{-3}\d h
\le 8\int_{\I\backslash E_\dt} (w_{hh\|}(t)+c_0)^{-3}\d h.
\end{equation}
Second, on
$E_\dt=\{w_{hh\|}(t) +c_0< \dt  \}$ we have $w_{hh\|}^\dt(t)=\dt+w_{hh\|}(t)$, so by \eqref{vphr} we know
\begin{equation}
w_{hh\|}^\dt(t) +\vep\vph_{hh}+c_0 =
w_{hh\|}(t)+\dt +\vep\vph_{hh}+c_0\ge w_{hh\|}(t)+c_0+\dt/2\ge (3/2)(w_{hh\|}(t)+c_0)
\label{low estimate}
\end{equation}
on
$E_\dt$ and
\begin{equation}\label{well-defined-2}
\int_{E_\dt} (w_{hh\|}^\dt(t)+\vep\vph_{hh}+c_0)^{-3}\d h
\le \int_{E_\dt} [(3/2)(w_{hh\|}(t)+c_0)]^{-3}\d h.
\end{equation}
Combining \eqref{well-defined-1}, \eqref{well-defined-2} and $(w_{hh\|}(t)+c_0)^{-3}\in L^1(\I)$
gives \eqref{well-defined}.

{\bf Step 3. Test with $v=w^\dt(t)\pm\vep \vph$.}

Plugging $v=w^\dt(t)+\vep \vph$ in \eqref{8vi} gives
\begin{equation}
\lg w_t(t),w^\dt(t)-w(t)+\vep\vph\rg+\phi(w^\dt(t)+\vep \vph)-\phi(w(t))\ge 0.
\label{9vi-2}
\end{equation}
Direct computation shows that
\begin{align*}
\phi(w^\dt(t)+\vep \vph)-\phi(w(t)) &= \frac12\int_\I\bigg[ \frac1{(w^\dt_{hh\|}(t)+\vep \vph_{hh}+c_0)^2}
-\frac1{(w_{hh\|}(t)+c_0)^2}\bigg]\d h\\
&=\frac12\int_\I
\frac{(w_{hh\|}(t)+c_0)^2-(w^\dt_{hh\|}(t)+\vep \vph_{hh}+c_0)^2}{(w^\dt_{hh\|}(t)+\vep \vph_{hh}+c_0)^2(w_{hh\|}(t)+c_0)^2}\d h\\
&=\frac12\int_\I
\frac{(w_{hh\|}(t)-w^\dt_{hh\|}(t)-\vep \vph_{hh})
(w_{hh\|}(t)+2c_0+w^\dt_{hh\|}(t)+\vep \vph_{hh})}{(w^\dt_{hh\|}(t)+\vep \vph_{hh}+c_0)^2(w_{hh\|}(t)+c_0)^2}\d h\\
&=\int_\I\frac{w_{hh\|}(t)-w^\dt_{hh\|}(t)-\vep \vph_{hh}}{(w^\dt_{hh\|}(t)+\vep \vph_{hh}+c_0)^2(w_{hh\|}(t)+c_0)}\d h\\
&\qquad -\int_\I\frac{(w_{hh\|}(t)-w^\dt_{hh\|}(t)-\vep \vph_{hh})^2}{2(w^\dt_{hh\|}(t)+\vep \vph_{hh}+c_0)^2(w_{hh\|}(t)+c_0)^2}\d h.
\end{align*}
This, together with \eqref{9vi-2}, gives
\begin{equation}
\lg w_t(t),w^\dt(t)-w(t)+\vep\vph\rg + \int_\I\frac{w_{hh\|}(t)-w^\dt_{hh\|}(t)-\vep \vph_{hh}}{(w^\dt_{hh\|}(t)+\vep \vph_{hh}+c_0)^2(w_{hh\|}(t)+c_0)}\d h\ge 0.\label{9vi-3}
\end{equation}
To take limit in \eqref{9vi-3}, we claim
\begin{align}
\lim_{\vep\to 0}\lg w_t(t),w^\dt(t)-w(t)+\vep\vph\rg/\vep &=\lg w_t(t),\vph\rg, \label{c-1}\\
\lim_{\vep\to 0}\int_\I\frac{w_{hh\|}(t)-w^\dt_{hh\|}(t)}{\vep(w^\dt_{hh\|}(t)+\vep \vph_{hh}+c_0)^2(w_{hh\|}(t)+c_0)}\d h &= 0
\label{c-2},\\
\lim_{\vep\to 0}\int_\I\frac{\vph_{hh}}{(w^\dt_{hh\|}(t)+\vep \vph_{hh}+c_0)^2(w_{hh\|}(t)+c_0)}\d h &=
\int_\I\frac{\vph_{hh}}{(w_{hh\|}(t)+c_0)^3}\d h.
\label{c-3}
\end{align}
{\em Proof of \eqref{c-1}.} Since $\lim_{\vep\to 0}\lg w_t(t),\vep\vph\rg/\vep =\lg w_t(t),\vph\rg$, thus it suffices to prove
\begin{equation*}
\lim_{\vep\to 0}\lg w_t(t),w^\dt(t)-w(t)\rg/\vep =0.
\end{equation*}
From the construction \eqref{trun01} we know $w^\dt_{hh\|}(t) = w_{hh\|}(t)+\dt\mathbf{1}_{E_\dt}$, so direct computation gives
\begin{align*}
\lim_{\vep\to 0}|\lg w_t(t),w^\dt(t)-w(t)\rg/\vep | &\le\lim_{\vep\to 0}\| w_t(t)\|_{U} \|w^\dt(t)-w(t)\|_{U}/\vep \\
&\le\lim_{\vep\to 0}\| w_t(t)\|_{U} \|w^\dt_{hh}(t)-w_{hh}(t)\|_{U}/\vep \\
&\le\lim_{\vep\to 0}\| w_t(t)\|_{U}  \dt|E_\dt|^{1/2}/\vep =0,
\end{align*}
where we used \eqref{tm911_1} and the relation \eqref{v} in the last equality.
Therefore \eqref{c-1} is proven.

\medskip

{\em Proof of \eqref{c-2}.} In view of \eqref{v}, recall that $w^\dt_{hh\|}(t) = w_{hh\|}(t)+\dt\mathbf{1}_{E_\dt}$,
and the relation $\dt/\vep=2\|\vph\|_{W^{2,\8}(\I)}+1$.  Hence
\begin{align*}
\int_\I\frac{w_{hh\|}(t)-w^\dt_{hh\|}(t)}{\vep(w^\dt_{hh\|}(t)+\vep \vph_{hh}+c_0)^2(w_{hh\|}(t)+c_0)}\d h
=\int_{\I}\frac{-(2\|\vph\|_{W^{2,\8}(\I)}+1)\mathbf{1}_{E_\dt}}{(w^\dt_{hh\|}(t)+\vep \vph_{hh}+c_0)^2(w_{hh\|}(t)+c_0)}\d h.
\end{align*}
By \eqref{low estimate} we also have
\begin{equation*}
\int_{\I}\frac{(2\|\vph\|_{W^{2,\8}(\I)}+1)\mathbf{1}_{E_\dt}}{(w^\dt_{hh\|}(t)+\vep \vph_{hh}+c_0)^2(w_{hh\|}(t)+c_0)}\d h
\le\int_{{E_\dt}}\frac{2\|\vph\|_{W^{2,\8}(\I)}+1}{(3/2)^2(w_{hh\|}(t)+c_0)^3}\d h\overset{\vep\to0}\to0,
\end{equation*}
where we have used $(w_{hh\|}(t)+c_0)^{-3}\in L^1(\I)$ by \eqref{L1}. Thus \eqref{c-2} is proven.

\medskip

{\em Proof of \eqref{c-3}.} From \eqref{upper estimate} and \eqref{low estimate}, we know
\begin{align*}
\frac{\vph_{hh}}{(w^\dt_{hh\|}(t)+\vep \vph_{hh}+c_0)^2(w_{hh\|}(t)+c_0)} &\to
\frac{\vph_{hh}}{(w_{hh\|}(t)+c_0)^3}\qquad \text{a.e. on } \I,\\
\frac{|\vph_{hh}|}{(w^\dt_{hh\|}(t)+\vep \vph_{hh}+c_0)^2(w_{hh\|}(t)+c_0)}
&\overset{\eqref{low estimate}}\le\frac{|\vph_{hh}|}{(3/2)^2(w_{hh\|}(t)+c_0)^3}\in L^1(\I) \qquad \text{on } E_\dt,\\
\frac{|\vph_{hh}|}{(w^\dt_{hh\|}(t)+\vep \vph_{hh}+c_0)^2(w_{hh\|}(t)+c_0)}
&\overset{\eqref{upper estimate}}\le\frac{|\vph_{hh}|}{(1/2)^2(w_{hh\|}(t)+c_0)^3} \in L^1(\I)\qquad \text{on } \I\backslash E_\dt,
\end{align*}
thus by Lebesgue's dominated convergence theorem we infer \eqref{c-3}.

\bigskip

Combining \eqref{c-1}, \eqref{c-2} and \eqref{c-3}, we can divide by $\vep>0$ in \eqref{9vi-3} and take the limit $\vep\to0^+$ to obtain
\begin{align*}
&\lim_{\vep\to0^+}\lg w_t(t),w^\dt(t)-w(t)+\vep\vph\rg/\vep +
\int_\I\frac{w_{hh\|}(t)-w^\dt_{hh\|}(t)-\vep \vph_{hh}}{\vep(w^\dt_{hh\|}(t)+\vep \vph_{hh}+c_0)^2(w_{hh\|}(t)+c_0)}\d h\\
&=\lg w_t(t),\vph\rg -
 \int_\I\frac{\vph_{hh}}{(w_{hh\|}(t)+c_0)^3}\d h\geq 0.
\end{align*}
Repeating the above arguments with $v=w^\dt(t)-\vep \vph$ gives
\begin{equation*}
\lg w_t(t),\vph\rg -
 \int_\I\frac{\vph_{hh}}{(w_{hh\|}(t)+c_0)^3}\d h\le0.
\end{equation*}
Thus we finally 
have
\begin{equation}
 \int_\I \bigg[w_t(t) -\big((w_{hh\|}(t)+c_0)^{-3}\big)_{hh}\bigg]\vph\d h =0\qquad \fal\vph\in C^2(\I),
\end{equation}
which gives
 $w_t(t)- [(w_{hh\|}(t)+c_0)^{-3}]_{hh}=0$ in $C^2(\I)'.$
 From the Radon-Nikodym theorem, we also know
$w_t(t)-[(w_{hh\|}(t)+c_0)^{-3}]_{hh}=0$ for a.e. $(t,x)\in[0,T]\times\I.$

   Finally, we turn to verify \eqref{dissiE}.
    Combining \eqref{main21_3} and \eqref{es1107_2}, we have the dissipation
    {law}
  \begin{equation}\label{tE}
    E(w(t))=\frac{1}{2}\|w_t\|_U^2=\frac{1}{2}\|((\eta +c_0)^{-3})_{hh}\|_U^2\leq \frac{1}{2}\|Bw_0\|_U^2=E(0)
  \end{equation}
{  for $E(w)=\frac{1}{2}\int_\I \big[((\eta +c_0)^{-3})_{hh}\big]^2 \d h$.}
  Hence the dissipation inequality \eqref{dissiE} holds and we complete the proof of Theorem \ref{mainth1}.
\qed\end{proof}

\section{acknowledgements}
We would like to thank the support by the National Science Foundation under Grant No. DMS-1514826 and KI-Net RNMS11-07444. We thank Jianfeng Lu for helpful discussions. Part of this work was carried out when Xin Yang Lu was affiliated with McGill University.

%

\end{document}